\documentclass[a4paper,10pt]{article}

\usepackage{graphicx}
\usepackage[a4paper, margin=2.5cm,top=3cm, bottom=3cm]{geometry}
\usepackage{amsthm}
\usepackage{amssymb}
\usepackage{amsmath}
\usepackage{hyperref}
\usepackage{physics}
\usepackage{xcolor}
\usepackage{enumerate}
\usepackage{enumitem}
\usepackage{listings}
\usepackage{complexity}
\usepackage{blkarray}
\usepackage{lmodern}

\usepackage[font=small,labelfont=bf]{caption}
\usepackage[font=small,labelfont=normalfont,labelformat=simple]{subcaption}
\graphicspath{{figs/}}

\newtheorem{theorem}{Theorem}[section]
\newtheorem{definition}[theorem]{Definition}
\newtheorem{lemma}[theorem]{Lemma}
\newtheorem{corollary}[theorem]{Corollary}
\newtheorem{question}{Question}

\newtheorem{proposition}[theorem]{Proposition}
\newtheorem{conjecture}{Conjecture}
\newtheorem{problem}{Problem}

\newcommand{\HG}{\mathrm{HG}}

\hypersetup{
	colorlinks,
	linkcolor={red!60!black},
	citecolor={green!50!black},
	urlcolor={blue!80!black}
}

\newcommand*\samethanks[1][\value{footnote}]{\footnotemark[#1]}

\author
{
Charlotte Knierim\thanks{Department of Computer Science, ETH Z\"urich, Switzerland\newline $\{$cknierim$\vert$anders.martinsson$\vert$ raphaelmario.steiner$\}$@inf.ethz.ch}

\and

Anders Martinsson \samethanks[1]

\and 

Raphael Steiner \samethanks[1] \thanks {Supported by an ETH Postdoctoral Fellowship.}
}

\date{\today}

\title{Hat guessing numbers of strongly degenerate graphs}

\begin{document}
\maketitle

\begin{abstract}
Assume $n$ players are placed on the $n$ vertices of a graph $G$. The following game was introduced by Winkler: An adversary puts a hat on each player, where each hat has a colour out of $q$ available colours. The players can see the hat of each of their neighbours in $G$, but not their own hat. Using a prediscussed guessing strategy, the players then simultaneously guess the colour of their hat. The players win if at least one of them guesses correctly, else the adversary wins. The largest integer $q$ such that there is a winning strategy for the players is denoted by $\HG(G)$, and this is called the \emph{hat guessing number} of $G$.

Although this game has received a lot of attention in the recent years, not much is known about how the hat guessing number relates to other graph parameters. For instance, a natural open question is whether the hat guessing number can be bounded from above in terms of degeneracy. In this paper, we prove that the hat guessing number of a graph can be bounded from above in terms of a related notion, which we call \emph{strong degeneracy}. We further give an exact characterisation of graphs with bounded strong degeneracy. As a consequence, we significantly improve the best known upper bound on the hat guessing number of outerplanar graphs from $2^{125000}$ to $40$, and further derive upper bounds on the hat guessing number for any class of $K_{2,s}$-free graphs with bounded expansion, such as the class of $C_4$-free planar graphs, more generally $K_{2,s}$-free graphs with bounded Hadwiger number or without a $K_t$-subdivision, and for Erd\H{o}s-R\'enyi random graphs with constant average degree.

\end{abstract}

\section{Introduction}
Multiple variants of hat guessing games have been studied in the past. They all have in common that an adversary puts coloured hats on the heads of a set of players, and some subset of the players need to be able to guess their own hat's colour correctly given only some partial information about the hat colours of certain other players (excluding themselves).  
In this paper we focus on a version introduced by Butler, Hajiaghayi and Kleinberg~\cite{butler2009hat}, which is sometimes also referred to as \emph{Winkler's hats game}, cf.~\cite{gadgeorg15}.

In this version, the $n$ players are placed on the vertices of an $n$-vertex graph $G$ and their vision is restricted to the players that sit on neighbouring vertices in $G$. Before the game starts, the players agree on a strategy how to make guesses for their hat colours. Formally, we can think of such a strategy as a family of functions $(f_v)_{v\in V(G)}$ where for every vertex $v\in V(G)$ the function $f_v$ describes the guess of the player sitting on vertex $v$ and $f_v\colon [q]^{N_G(v)} \rightarrow [q]$ where $q$ is the number of colours for the hats (known to the players).
The adversary then chooses which hat to give to which player (while knowing the strategy $(f_v)_{v\in V(G)}$). This adversarial colouring can be described as a mapping $c\colon V(G)\to [q]$. We say a strategy $(f_v)_{v\in V(G)}$ is \emph{winning} if for all adversarial colourings $c$, there it at least one vertex $v\in V(G)$ such that $f_v((c(w))_{w \in N_G(v)})=c(v)$ (i.e., at least one vertex correctly guesses the colour of its hat). 
We denote by $\HG(G)$ the \emph{hat guessing number} of $G$, which is defined as the largest integer $q$ such that there exists a strategy $(f_v)_{v\in V(G)}$ that is winning on $q$ colours.

In general, it seems that both deriving good lower bounds and good upper bounds on the hat guessing number, even of very specific graphs, is extremely challenging. Accordingly, in the literature there are only very few examples and classes of graphs $G$ for which $\HG(G)$ has been determined exactly. For example, the cliques $K_n$ are know to have hat guessing number exactly $n$ (compare for instance~\cite{feige2004you}), but for the complete bipartite graphs $K_{m,n}$ determining even only the correct asymptotic growth of their hat guessing numbers remains an open problem. For instance, for $m=n$ the best known upper bound on $\HG(K_{n,n})$ is $n+1$, but the best known lower bound proved by Alon et al.~\cite{alon2020hat} is of magnitude $\Omega(n^{1/2-o(1)})$. Butler et al. proved in~\cite{butler2009hat} (see also~\cite{alon2020hat,bosek2021hat}) that all trees on at least two vertices have hat guessing number $2$, and also the hat guessing numbers of cycles have been determined exactly by Szczechla in~\cite{szczechla2017three}\footnote{more precisely, it is known that $\HG(C_n)=3$ if $n$ is equal to $4$ or divisible by $3$, and $\HG(C_n)=2$ otherwise}.
Furthermore, using a random colouring approach for the adversary, it can be shown via the Lov\'{a}sz local lemma that the hat guessing number of all graphs with maximum degree at most $\Delta$ is bounded from above by $e\Delta$~\cite{farnik2015hat}. More bounds on hat guessing numbers for specific graph classes can be found in~\cite{bosek2021hat,he2020hat,he2020hat2}.

In order to gain a qualitative understanding of the complex parameter $\text{HG}(G)$, it is natural to relate it to other known graph parameters. Intuitively, on very sparse graphs, such as graphs which are $d$-degenerate for a constant $d \in \mathbb{N}$, the players should not be able to win the game played with an arbitrarily large number of colours. Vice versa, possibly every graph with a sufficiently high minimum degree allows for winning the guessing game with many colours. This intuition has inspired the following two open problems which were raised in~\cite{alon2020hat} and~\cite{farnik2015hat}.
\begin{problem}[cf.~Problem 1.4 in~\cite{alon2020hat}]\label{problems}
Do there exist functions $f_1, f_2:\mathbb{N} \rightarrow \mathbb{N}$ such that the following hold?
\begin{enumerate}
    \item If a graph $G$ is $d$-degenerate, then $\text{HG}(G) \le f_1(d)$.
    \item If a graph $G$ has minimum degree at least $f_2(d)$, then $\text{HG}(G) \ge d$.
\end{enumerate}
\end{problem}

It was proved in~\cite{he2020hat2} that for every $d \in \mathbb{N}$ there exist $d$-degenerate graphs with hat guessing number as large as $2^{2^{d-1}}$, which shows that $f_1$, if it exists, must have at least a doubly-exponential growth.
If both of the statements in Problem~\ref{problems} held true, this would mean that the hat guessing number of graphs is an equivalent parameter to the degeneracy of the graph, which would substantially enhance our structural understanding of the hat guessing game. However, the above problems remain widely open.

In this paper, we are motivated in particular by Problem (1.) above.

We provide positive evidence, by showing that the hat guessing number of a graphs can bounded from above in terms of what we call \emph{strong degeneracy}. This is a strengthened variant of the usual degeneracy notion for graphs, defined as follows.

\begin{definition}\label{def:strongdeg}
Let $d \ge 1$ be an integer and $G$ a graph. We say that a vertex $v \in V(G)$ is \emph{$d$-removable in $G$} if $d_G(v) \le d$ and if $|\{w \in N_G(v)\vert d_G(w)>d\}| \le 1$.
We say that a graph $G$ is \emph{strongly $d$-degenerate} if every non-empty subgraph $G'$ of $G$ contains a vertex which is $d$-removable in $G'$. 
\end{definition}

Definition~\ref{def:strongdeg} is reminiscent of the usual degeneracy notion: A graph $G$ is $d$-degenerate if and only if every non-empty subgraph $G'$ of $G$ contains a vertex of degree at most $d$. Comparing this to Definition~\ref{def:strongdeg} we can easily see that every strongly $d$-degenerate graph is also $d$-degenerate. However, the same is not true in reverse. For a simple example, consider the complete bipartite graph $K_{2,s}$ for some integer $s \in \mathbb{N}$. This graph is clearly $2$-degenerate, but contains no $(s-1)$-removable vertex and is therefore not strongly $d$-degenerate for any $d<s$.
The following is our main result. 

\begin{theorem}\label{thm:hattiness}
Let $d \in \mathbb{N}$. Then every strongly $d$-degenerate graph $G$ satisfies $\HG(G) \le (2d)^d$.
\end{theorem}

Our proof of Theorem~\ref{thm:hattiness} is inductive and relies on an appropriate extension of the hat guessing game in which each player may guess more than one colour for its hat at once. A similar extension was previously analysed by Bosek et al.~\cite{bosek2021hat}. The idea is that any $d$-removable vertex has the property that removing it from the graph and increasing the number of guesses allowed by its low degree neighbours appropriately does not decrease the hat guessing number. By iterating this process, it follows that the hat guessing problem on any strongly $d$-degenerate graph is at most as hard as an isolated vertex with $(2d)^d$ guesses.

While the strong degeneracy notion might seem a bit unnatural at first, in Section~\ref{sec:boundeddeg}, we will establish the following exact characterisation of graph classes that have bounded strong degeneracy. As illustrated further below, this characterisation reveals many interesting sparse graph classes for which Theorem~\ref{thm:hattiness} yields bounds on the hat guessing numbers.

\begin{theorem}\label{thm:char}
Let $\mathcal{G}$ be a class of graphs closed under taking subgraphs. The following are equivalent. 
\begin{itemize}
    \item There exists $d \in \mathbb{N}$ such that every graph $G \in \mathcal{G}$ is strongly $d$-degenerate. 
    \item There exist $s \in \mathbb{N}$ and $k \in \mathbb{N}$ such that $K_{2,s} \notin \mathcal{G}$ and such that for every graph $H$ of minimum degree at least $k$, the $1$-subdivision\footnote{Recall that the $1$-subdivision of $H$ is the graph obtained by subdividing every edge of $H$ with a new vertex, transforming it into a path of length $2$ connecting the same endpoints.} of $H$ is not contained in $\mathcal{G}$. 
\end{itemize}
\end{theorem}

In order to prove Theorem~\ref{thm:char}, in Section~\ref{sec:boundeddeg}, we bound the strong degeneracy of a graph in terms of its so-called \emph{top-grads}, which are certain density-measures for graphs introduced by Ne\v{s}etril and Ossona de Mendez. We defer the details of these bounds to Section~\ref{sec:boundeddeg}. 

Let us now mention some applications of our results to concrete graph classes. 

\begin{itemize}
    \item It is immediate from the definition that every tree is strongly $1$-degenerate. Hence, a special case of Theorem~\ref{thm:hattiness} reproves the previously known result (cf.~\cite{alon2020hat,bosek2021hat,butler2009hat}) that every tree $T$ satifies $\text{HG}(T) \le 2$. 

    \item It remains an open problem (stated explicitly in~\cite{alon2021hat,bosek2021hat,bradshaw}), whether or not planar graphs allow for a constant bound on their hat guessing number. The best known lower bound for the hat guessing number of planar graphs is $14$, as proved in~\cite{kokhas2021hats}. Approaching this problem, Bradshaw showed in~\cite{bradshaw} that outerplanar graphs\footnote{A graph is called \emph{outerplanar} if it can be drawn in the plane without crossings, such that all vertices touch the unbounded outer region of the drawing.} have hat guessing number at most $2^{125000}$, and suspected that there exist outerplanar graphs with hat guessing number as large as $1000$. In this paper, we drastically improve Bradshaw's bound. 
    \begin{theorem}\label{thm:outplanar}
    If $G$ is an outerplanar graph, then $\HG(G) \le 40$. \end{theorem}
    
    It is not hard to show that every outerplanar graph is strongly $4$-degenerate, which by Theorem~\ref{thm:hattiness} would imply a bound of $8^4=2^{12}$ on the hat guessing number. However, a slightly more careful analysis which is given in Section~\ref{sec:boundeddeg} yields the better bound stated in Theorem~\ref{thm:outplanar}.
    
    \item Another result which follows from Theorem~\ref{thm:char} is that there exists a constant $d$ such that every $C_4$-free planar graph is strongly $d$-degenerate. Indeed, since $K_{2,2}=C_4$ we can put $s=2$ in Theorem~\ref{thm:char}, and with $k=6$ we can see that since every graph $H$ of minimum degree at least $k$ is non-planar, also the $1$-subdivision of $H$ is non-planar. 
    
    In fact, by combining ideas from Section \ref{sec:boundeddeg} with double counting using Eulers formula, it can be shown that any $C_4$-free planar graph is strongly $7$-degenerate, which implies $\text{HG}(G) \le 14^7 \sim 10^8$ for every $C_4$-free planar graph $G$. This qualitatively strengthens a result of Bosek et al.~\cite{bosek2021hat} that planar graphs of girth at least $14$ have bounded hat guessing number.
    
    \item Bosek et al.~\cite{bosek2021hat} proved that for every fixed integer $\gamma$, graphs embeddable on a surface of Euler-genus $\gamma$ and of sufficiently large girth (in terms of $\gamma$) have hat guessing number at most $6$. 
    
    As a consequence of Theorem~\ref{thm:char}, $C_4$-free graphs embeddable on a surface of Euler-genus $\gamma$ have bounded strong degeneracy and hence bounded hat guessing numbers (in terms of $\gamma$). This somewhat complements Bosek et al.'s result, trading the constant bound of $6$ for the hat guessing number for a much weaker requirement on the graph (girth $5$ suffices). The reason why we can apply Theorem~\ref{thm:char} to $C_4$-free graphs with Euler-genus at most $\gamma$ is that graphs of bounded genus have bounded degeneracy, and hence, for every $\gamma$ there exists $k=k(\gamma) \in \mathbb{N}$ such that every graph $H$ of minimum degree at least $k$ cannot be embedded on a surface of Euler-genus $\gamma$, and in particular, the same is true for the $1$-subdivision of $H$.
    
    \item An open problem raised by Alon and Chizewer in~\cite{alon2021hat} is whether the hat guessing number of graphs can be bounded in terms of their \emph{Hadwiger number}. The Hadwiger number $h(G)$ of a graph $G$ is defined as the largest integer $t \ge 1$ such that $G$ contains a $K_t$-minor. Using Theorem~\ref{thm:char}, we can give a positive answer to this question for graphs without $K_{2,s}$-subgraphs, for any fixed $s \in \mathbb{N}$, for instance, for $C_4$-free graphs. More concretely, we obtain the following result in Section~\ref{sec:boundeddeg}.
    
    \begin{corollary}\label{hadwiger}
     Every $K_{2,s}$-free graph with no $K_t$-minor is strongly $d$-degenerate, where $d=d(s,t)=O(st\sqrt{\log t})$. In particular, $\text{HG}(G) \le 2^{O(st\sqrt{\log t}\log(st))}$ for every such graph. 
    \end{corollary}
    
    We remark that it was observed in~\cite{alon2021hat} that there exist graphs with no $K_t$-minor and whose hat guessing number is equal to $2^{2^{t-3}}$.
    
    \item Qualitatively generalising Corollary~\ref{hadwiger}, in Section~\ref{sec:boundeddeg} we also prove a similar result for graphs with no \emph{topological} $K_t$-minor, i.e., graphs which do not contain a subdivision of $K_t$ as a subgraph. \begin{corollary}\label{topological}
     Every $K_{2,s}$-free graph containing no $K_t$-subdivision is strongly $d$-degenerate, where $d=d(s,t)=O(st^4)$. In particular, $\text{HG}(G) \le 2^{O(st^4\log(st))}$ for every such graph. 
    \end{corollary}
    \item The most general application, which (at least qualitatively) also subsumes all the previously listed results, addresses graph classes of \emph{bounded expansion}. These form a central concept in the \emph{sparsity theory} of graph classes developed by Ne\v{s}etril and Ossona de Mendez. Many of the cornerstone results on this topic are summarised in the book~\cite{sparsity} by Ne\v{s}etril and Ossona de Mendez. We give detailed definitions in Section~\ref{sec:boundeddeg}, here let it only be said that most natural graph classes of bounded degeneracy fall under the umbrella of bounded expansion, including (but not limited to) graphs of bounded maximum degree, graphs of bounded genus, graphs containing no $K_t$-minor, graphs containing no $K_t$-subdivision, string graphs with a forbidden complete bipartite subgraph, and sparse random graphs. As a consequence of Theorems~\ref{thm:hattiness} and~\ref{thm:char} we can prove that graphs belonging to any one of these classes and which do not contain a pair of vertices sharing $s$ neighbours for a fixed constant $s \ge 1$, have bounded hat guessing number. 
\begin{corollary}\label{boundedexpansion}
Let $\mathcal{G}$ be a class of graphs with bounded expansion, and let $s \in \mathbb{N}$. Then there exists an integer $d=d(\mathcal{G},s)$ such that every $K_{2,s}$-free graph $G \in \mathcal{G}$ is strongly $d$-degenerate and hence satisfies $\text{HG}(G) \le (2d)^d$.
\end{corollary}

\item Alon and Chizewer~\cite{alon2021hat} initiated the study of hat guessing on the Erd\H{o}s-Renyi-random graphs $G(n,p)$, and showed that with high probability\footnote{or w.h.p.\ for short, with probability tending to one as $n\to \infty$.} $\HG(G(n,1/2))\ge n^{1-o(1)}$. While their result addresses very dense random graphs, as an application of the previously mentioned theory of classes with bounded expansion we show in Section~\ref{sec:boundeddeg} that on the opposite side of the spectrum, very sparse Erd\H{o}s-Renyi-random graphs with a constant average degree have a constant hat guessing number, providing further evidence towards a positive answer for Problem~\ref{problems}. 
\begin{corollary}\label{random}
There exists a function $f:\mathbb{R}_+ \rightarrow \mathbb{N}$ such that for every fixed $C>0$, it holds w.h.p.\ as $n \rightarrow \infty$ that 
\[\HG\left(G\left(n,\frac{C}{n}\right)\right) \le f(C).\]
\end{corollary}
\end{itemize}

\paragraph{Organization.} The rest of this paper is structured as follows. In Section~\ref{sec:hatnumber} we present a generalized variant of the hat guessing game, which is then used in the proof of Theorem~\ref{thm:hattiness}. We conclude the section by proving the upper bound on the hat guessing number of outerplanar graphs in Theorem~\ref{thm:outplanar}.

In Section~\ref{sec:boundeddeg} we give upper bounds on the strong degeneracy of a graph in terms of its top-grads. Using these bounds, we prove Theorem~\ref{thm:char}.  Finally, we deduce the consequences of these results for the sparse graph classes mentioned above, including the proofs of Corollaries~\ref{hadwiger},~\ref{topological},~\ref{boundedexpansion} and~\ref{random}.

\section{Bounding the hat guessing number in terms of strong degeneracy}\label{sec:hatnumber}
In order to analyse the hat guessing number in terms of the strong degeneracy of a graph $G$, we propose a slight variant of the original game. A variant similar to this has also been considered in~\cite{bosek2021hat}.  
Let $G$ be a graph, and let $g:V(G) \rightarrow \mathbb{N}$ be an assignment of positive integers to its vertices. The multiple-guess hats-game with respect to $g$ and with $q$ colours is played just as the ordinary hat guessing game, only that every vertex $v$ is allowed to guess up to $g(v)$ colours at a time, and the players win the game if for every adversarial colouring there is a vertex $v$ for which its assigned colour is contained in the list of $\le g(v)$ colours that it has guessed. Formally, the strategy of the players consists of functions $(f_v)_{v \in V(G)}$ where for every $v \in V(G)$ $f_v$ is a mapping $f:[q]^{N_G(v)} \rightarrow \binom{[q]}{\le g(v)}$ assigning to every possible $q$-colouring of its neighbourhood a set of at most $g(v)$ colours in $[q]$.

Define $\text{HG}_g(G)$ as the largest number $q$ for which the players have a winning strategy when playing the game described above (with respect to $g$) if $q$ colours are used in total. 
Formally, we say that a strategy $(f_v)_{v \in V(G)}$ is \emph{winning} if for every colour-assignment $c:V(G) \rightarrow [q]$ there exists $v \in V(G)$ such that $c(v) \in f_v((c(w))_{w \in N_G(v)})$. 
Trivially, $\text{HG}_g(G)=\text{HG}(G)$ if $g \equiv 1$, and $\text{HG}_{g_1}(G) \le \text{HG}_{g_2}(G)$ for every pair $g_1,g_2$ of assignments such that $g_1 \le g_2$. 

We prove Theorem~\ref{thm:hattiness} by means of the following more general statement. 
\begin{proposition}\label{guessnumberchange}
Let $G$ be a graph with an assignment $g:V(G) \rightarrow \mathbb{N}$, and let $v \in V(G)$. Denote $G':=G-v$ and let $g':V(G')\rightarrow \mathbb{N}$ be such that $g'(x)=g(x)$ for every $x \in V(G)\setminus(\{v\} \cup N_G(v))$. If for some integer $q>HG_{g'}(G')$ we have
$$\frac{g(v)}{q}+\sum_{w \in N_G(v)}{\frac{g(w)}{g'(w)+1}}<1,$$ then $q>\text{HG}_{g}(G)$. 
\end{proposition}
\begin{proof}
Let $q>HG_{g'}(G')$ be such that
\begin{equation}\label{eq0}\frac{g(v)}{q}+\sum_{w \in N_G(v)}{\frac{g(w)}{g'(w)+1}}<1.\end{equation}
Note that $\text{HG}_{g'}(G')<q$ also implies $g'(w)<q$ for all $w \in V(G)\setminus \{v\}$. 
In order to prove the claim that $\text{HG}_g(G)<q$, we have to show that if the hat guessing game with respect to $g$ and with $q$ colours is played on $G$, then the players do not have a winning strategy. 
Towards a contradiction, suppose that there exists a winning strategy $(f_w)_{w \in V(G)}$ for the multiple-guess hats-game on the graph $G$, played with colour-set $[q]$ and with respect to the multiplicities given by $g$. 

Let us define a modified set of guessing functions on the reduced graph $G'$, which we denote by $(f_w')_{w \in V(G')}$, as described in the following. For every $w \in V(G')=V(G)\setminus\{v\}$, the function $f_w'$ maps from $[q]^{N_{G'}(w)}$ to $\binom{[q]}{\le g'(w)}$. 

Firstly, for every $w \in V(G)\setminus(\{v\} \cup N_G(v))$, we let $f_w':=f_w$. 

Secondly, consider a fixed vertex $w \in N_G(v)$ and a given colouring $(c_x)_{x \in N_{G'}(w)}$ of its neighbourhood in $G'$. In the following we give a definition of $f_w'((c_x)_{x \in N_{G}(w)\setminus\{v\}})$.
To do so, consider a discrete probability space in which we colour the vertex $v$ with a colour $c \in [q]$, chosen uniformly at random. For every colour $i \in [q]$, let $A_w^i$ denote the event that ``$f_w((c_x)_{x \in N_{G}(w)\setminus\{v\}},c) \ni i$'', and let $p_w^i:=\mathbb{P}(A_w^i)$ for every $i \in [q]$. Finally, consider a linear order on $[q]$ extending the partial ordering given by the values $p_w^i, i \in [q]$, and let $I \subseteq [q]$ be the set of $g'(w)$ largest colours with respect to this partial order. We define \[f_w'((c_x)_{x \in N_{G}(w)\setminus\{v\}}):=I \in \binom{[q]}{\le g'(w)}.\] Note that this is a well-defined strategy, since all the probabilities $p_w^1,\ldots,p_w^q$ can be calculated and hence also the choice of $I$ can be made, solely dependent on $f_w$ and the colouring $(c_x)_{x \in N_{G'}(w)}$ of the neighbourhood of $w$ in $G'$.

Further note that for every $w \in N_G(v)$ and every colouring $(c_x)_{x \in N_{G'}(w)}$ of its neighbourhood in $G'$, by definition of $I$ we have that for every $j \in [q]\setminus f_w'((c_x)_{x \in N_{G'}(w)})$, it holds that
\begin{equation}\label{eq1} 
\begin{split}
\mathbb{P}(j \in f_w((c_x)_{x \in N_{G}(w)\setminus\{v\}},c))&=p_w^j \le \frac{1}{|I|+1}\sum_{i \in I \cup \{j\}}{p_w^i} \le \frac{1}{|I|+1}\sum_{i \in [q]}{p_w^i} \\ &=\frac{1}{g'(w)+1}\mathbb{E}(|f_w((c_x)_{x \in N_{G}(w)\setminus\{v\}},c)|) \le \frac{g(w)}{g'(w)+1}.
\end{split}
\end{equation}

We are now ready to lead our initial assumption towards a contradiction, by constructing a colouring of $G$ for which the strategy $(f_w)_{w \in V(G)}$ does not win (that is, no vertex guesses a set of colours containing its actual colour). In the following, we call such a colouring \emph{mean}.
Since $q>\text{HG}_{g'}(G')$ by our inductive assumption, there exists a mean colouring of $G'$ with respect to the strategies $(f_w')_{w \in V(G')}$, that is, a mapping $c:V(G') \rightarrow [q]$ such that $c(w) \notin f_{w}'((c(x))_{x \in N_{G'}(w)})$, for \emph{every} $w \in V(G')$. 

\paragraph{Claim.} There is $c \in [q]$ such that $c \notin f_v((c(w))_{w \in N_G(v)})$ and $c(w) \notin f_w((c(x))_{x \in N_G(w)\setminus \{v\}},c)$, for every $w \in N_G(v)$. 
\begin{proof}[Subproof]
The claim is trivially true if $d_G(v)=0$, since we then may simply pick $c$ as a colour which is not contained in the guessing set of $v$ determined by $f_v$ (note that the latter set has size $g(v)<q$, since $\frac{g(v)}{q}<1$ by \eqref{eq0}).

Moving on, suppose that $d_G(v) \ge 1$.
Let us choose a colour $c \in [q]$ uniformly at random, and show that it has the required properties with positive probability. Let $E$ denote the event that one of the two properties stated in the claim does not hold. $E$ is covered by the union of the events ``$c \in f_v((c(w))_{w \in N_G(v)})$'' and ``$c(w) \in f_w((c(x))_{x \in N_G(w) \setminus \{v\}},c)$'' for $w \in N_G(v)$. 

Clearly, we have \[\mathbb{P}(c \in f_v((c(w))_{w \in N_G(v)}))=\frac{|f_v((c(w))_{w \in N_G(v)})|}{q} \le \frac{g(v)}{q}.\] Furthermore, for every $w \in N_G(v)$, we have that $j:=c(w) \notin f_{w}'((c(x))_{x \in N_{G'}(w)})$. By inequality~\ref{eq1}, we have
\[\mathbb{P}(c(w) \in f_w((c(x))_{x \in N_G(w) \setminus \{v\}},c)) \le \frac{g(w)}{g'(w)+1},\] for every $w \in N_G(v)$. Applying a union bound, and making use of inequality~\ref{eq0}, we have
$$\mathbb{P}(E) \le \frac{g(v)}{q}+\sum_{w \in N_G(v)}{\frac{g(w)}{g'(w)+1}}<1.$$ 
Consequently, we have $\mathbb{P}(\bar{E})>0$, proving the existence of $c$ with the desired properties.
\end{proof}
We now consider the $[q]$-colouring of $G$ obtained by extending the colouring $c(\cdot)$ of $G'$ with the colour $c$ from the above claim at $v$. We claim that this is a mean colouring of $G$ with respect to $(f_w)_{w \in V(G)}$ (which then concludes the proof). Indeed, for every $w \in V(G)\setminus (\{v\} \cup N_G(v))$, we have $f_w'=f_w$ by definition, and hence (since $c(\cdot)$ is a mean colouring of $G'$) $c(w) \notin f_w'((c(x))_{x \in N_{G'}(w)})=f_w((c(x))_{x \in N_{G}(w)})$. For $w \in N_G(v)$ and the vertex $v$, the properties of $c$ guaranteed by the claim state exactly that their guessed colour-set does not contain their actual colour. Hence, this colouring is indeed mean, concluding the proof by contradiction.
\end{proof}
Using Proposition~\ref{guessnumberchange}, we can now rather easily derive Theorem~\ref{thm:hattiness}.
\begin{proof}[Proof of Theorem~\ref{thm:hattiness}]

We prove the stronger claim that for every strongly $d$-degenerate graph $G$, we have $\text{HG}_{g}(G) \le (2d)^d$, where $g:V(G) \rightarrow \mathbb{N}$ is defined by
$$g(v):=\begin{cases} 1, & \text{if }d_G(v)>d, \cr (2d)^{d-d_G(v)}, & \text{if }d_G(v) \le d \end{cases}.$$
We prove this stronger claim by induction on the number of vertices of $G$. If $v(G)=1$, denote $V(G)=\{v\}$ and note that $g(v)=(2d)^d$. The fact that $\text{HG}_g(G) \le (2d)^d$ now follows trivially (if $q>(2d)^d$, the guess of $v$ is a set of $\le (2d)^d$ colours which does not change when changing the colour assigned to $v$, hence the adversary may just put a colour on $v$ which does not appear in its guess). 

For the induction step, suppose that $v(G) \ge 2$ and the claim holds for every strongly $d$-degenerate graph on fewer than $v(G)$ vertices.
Since $G$ is strongly $d$-degenerate, it contains a $d$-removable vertex $v$. Unless $d_G(v)=0$, we denote by $u \in N_G(v)$ a vertex such that $d_G(w) \le d$ for every $w \in N_G(v) \setminus \{u\}$. 

Put $G':=G-v$ and note that $G'$ as a subgraph of $G$ is still strongly $d$-degenerate. Hence, we may apply the induction hypothesis to $G'$ and we find that $\text{HG}_{g'}(G) \le (2d)^d$, where $g':V(G)\setminus\{v\} \rightarrow \mathbb{N}$ is defined by 
$$g'(w):=\begin{cases} 1, & \text{if }d_{G'}(w)>d, \cr (2d)^{d-d_{G'}(w)}, & \text{if }d_{G'}(w) \le d \end{cases}.$$
Note that $g'(w)=g(w)$ for every $w \in V(G)\setminus (\{v\} \cup N_G(v))$ and $g'(w)=2d\cdot g(w)$ for every $w \in N_G(v)\setminus\{u\}$.

In order to prove the inductive claim that $\text{HG}_g(G) \le (2d)^d$, we invoke Proposition~\ref{guessnumberchange} and show that for $q:=(2d)^d+1>\text{HG}_{g'}(G')$ the inequality
\begin{equation}\label{eq2}\frac{g(v)}{q}+\sum_{w \in N_G(v)}{\frac{g(w)}{g'(w)+1}}<1\end{equation} is satisfied, which then yields that $\text{HG}_g(G) \le q-1 = (2d)^d$, as required. 

Since $d_G(v) \le d$ ($v$ is a $d$-removable vertex of $G$), we have $g(v)=(2d)^{d-d_G(v)}$ by definition. This directly implies that \eqref{eq2} is satisfied if $d_G(v)=0$, hence, in the following assume that $d_G(v) \ge 1$ (and hence that $u$ exists). 
For every $w \in N_G(v)\setminus\{u\}$, from the above we have \[\frac{g(w)}{g'(w)+1}<\frac{g(w)}{g'(w)}=\frac{g(w)}{2dg(w)}=\frac{1}{2d}.\] Finally, observe that by definition of $g$ and $g'$ we either have $d_{G}(u) \le d$ and hence $g'(u)=2d\cdot g(u)$, or $d_G(u)>d$ and thus $g(u)=g'(u)=1$. In both cases, we immediately obtain $\frac{g(u)}{g'(u)+1} \le \frac{1}{2}$. Hence, we may upper-bound the left hand side of \eqref{eq2} by

$$\frac{(2d)^{d-d_G(v)}}{(2d)^d+1}+\frac{1}{2}+(d_G(v)-1)\frac{1}{2d}$$
The real-valued function $x \rightarrow \frac{(2d)^{d-x}}{(2d)^d+1}+\frac{1}{2}+(x-1)\frac{1}{2d}$ is easily seen to be convex, and hence achieves its maximum on $[1,d]$ at a value in $\{1,d\}$. We conclude that
$$\frac{g(v)}{q}+\sum_{w \in N_G(v)}{\frac{g(w)}{g'(w)+1}} \le \max\left\{\frac{(2d)^{d-1}}{(2d)^d+1}+\frac{1}{2},1+\frac{1}{(2d)^d+1}-\frac{1}{2d}\right\}<1,$$ as required. This concludes the proof of the theorem.
\end{proof}

We conclude this section by establishing Theorem~\ref{thm:outplanar}. We prepare the proof with an elementary fact about so-called \emph{maximal outerplanar graphs}. A graph $G$ on at least $3$ vertices is called \emph{maximal outerplanar} if it is outerplanar and the addition of any edge between two non-adjacent vertices in $G$ results in a graph which is not outerplanar. 

\begin{lemma}\label{lem:outer}
Let $G$ be a maximal outerplanar graph on more than $3$ vertices. Then there exists a vertex $v \in V(G)$ such that $d_G(v)=2$, $G-v$ is a maximal outerplanar graph, and such that at least one neighbour of $v$ has degree at most $4$ in $G$.
\end{lemma}
\begin{proof}
Consider an outerplanar embedding of $G$, i.e., such that all of its vertices are incident to the unbounded region of the embedding. Since $G$ is maximal outerplanar, the outer face of $G$ forms a Hamiltonian cycle $C$ of $G$, and we can label the vertices of $G$ as $v_1,\ldots,v_n$ in circular order according to their appearance on the outer cycle $C$. Furthermore, the maximality of $G$ implies that every inner face of $G$ forms a triangle, hence, $G$ is a triangulation of the cycle $C$ without interior vertices. Let $X$ be the set of degree $2$-vertices in $G$. Note that $X$ is non-empty: To see this consider a chord of $C$ minimising the distance of its endpoints on $C$. Then this edge is of the form $v_{i-1}v_{i+1}$ for some $i \in \{1,\ldots,n\}$ (addition and subtraction modulo $n$), and we can see that $v_i \in X$. 

Furthermore, pause to note that every vertex in $V(G)\setminus X$ is adjacent to at most $2$ vertices in $X$. Now consider the subgraph $H:=G-X$ obtained by removing all degree $2$-vertices. Clearly the graph $H$ is still outerplanar and hence a spanning subgraph of a maximal outerplanar graph, which (following the above argumentation) contains a vertex of degree at most $2$. Hence, also in $H$  there is a vertex $u$ satisfying $d_H(u) \le 2$. Clearly, $u$ must be adjacent in $G$ to at least one vertex in $X$, for otherwise it would have had degree at most $2$ also in $G$, meaning that $u \in X$, a contradiction. Let $v \in X$ be a neighbour of $u$ in $G$. By the above, $u$ is adjacent to at most two vertices in $X$, which means that $d_G(u) \le d_H(u)+2 \le 4$. Since $v$ is of degree $2$ in $G$ the graph $G-v$ is still maximal outerplanar, and the claim follows.
\end{proof}

\begin{proof}[Proof of Theorem~\ref{thm:outplanar}]
Since the hat guessing number of a graph is monotone under taking subgraphs, it will be sufficient to prove the claim of the theorem for maximal outerplanar graphs.

Using Proposition~\ref{guessnumberchange}, we will show the following slightly stronger statement for all maximal outerplanar graphs (note that a maximal outerplanar graph has no vertices of degree $<2$): 
\paragraph{Claim.} If $G$ is a maximal outerplanar graph, and if $g:V(G) \rightarrow \mathbb{N}$ is defined by $$g(w):=\begin{cases}1, & \text{if }d_G(w) \ge 4, \cr 2, & \text{if }d_G(w)=3, \cr 4, & \text{if }d_G(w)=2, \end{cases}$$ then $\text{HG}_g(G) \le 40$.

\bigskip

The rest of the proof is devoted to verifying the above claim. We do this by induction of the number $v(G)$ of vertices of the maximal outerplanar graph $G$. If $v(G)=3$, then $G=K_3$, so we have $g \equiv 4$, and consequently $HG_g(G) \le 12 \le 40$.
For the inductive step, assume $v(G) \ge 4$ and that the statement of the claim holds for maximal outerplanar graphs on at most $v(G)-1$ vertices. By Lemma~\ref{lem:outer}, if $G$ is a maximal outerplanar graph on more than $3$ vertices, then there exists a vertex $v \in V(G)$ such that $d_G(v)=2$, $G':=G-v$ is a maximal outerplanar graph, and such that at least one neighbour of $v$ has degree at most $4$ in $G$. Pick such a vertex $v$, and denote its neighbours by $w_1, w_2$, such that $d_G(w_1) \le 4$. By the induction hypothesis, we have $\text{HG}_{g'}(G') \le 40$, where $g':V(G)\setminus \{v\} \rightarrow \mathbb{N}$ is defined by

$$g'(w):=\begin{cases}1, & \text{if }d_{G'}(w) \ge 4, \cr 2, & \text{if }d_{G'}(w)=3, \cr 4, & \text{if }d_{G'}(w)=2, \end{cases}$$

Let $q:=41$ and let us show, using Proposition~\ref{guessnumberchange}, that $\text{HG}_g(G)<q$ (which will then prove the inductive claim). According to the statement of the Proposition, it suffices to verify the following inequality:
$$\frac{g(v)}{41}+\frac{g(w_1)}{g'(w_1)+1}+\frac{g(w_2)}{g'(w_2)+1}<1.$$

By definition, we have $g(v)=4$. 
Note that since $G-v$ is maximal outerplanar, $w_1$ has degree at least two in $G-v$, and hence $d_G(w_1) \in \{3,4\}$. If $d_G(w_1)=3$, then $d_{G'}(w_1)=2$, and hence $\frac{g(w_1)}{g'(w_1)+1}=\frac{2}{5}$. If $d_G(w_1)=4$, then $\frac{g(w_1)}{g'(w_1)+1}=\frac{1}{3}<\frac{2}{5}$, so $\frac{g(w_1)}{g'(w_1)+1} \le \frac{2}{5}$ in any case.

Finally, if $d_{G'}(w_2) \ge 4$, then $g(w_2)=g'(w_2)=1$ by definition, and hence $\frac{g(w_2)}{g'(w_2)+1}=\frac{1}{2}$. Otherwise, if $d_{G'}(w_2) \le 3$, then $g'(w_2)=2g(w_2)$ (since $d_G(w_2)=d_{G'}(w_2)+1$). This yields $\frac{g(w_2)}{g'(w_2)+1} <\frac{g(w_2)}{2g(w_2)}=\frac{1}{2}$. In every case, we may conclude that
$$\frac{g(v)}{41}+\frac{g(w_1)}{g'(w_1)+1}+\frac{g(w_2)}{g'(w_2)+1} \le \frac{4}{41}+\frac{2}{5}+\frac{1}{2}<1,$$ and this concludes the proof of the claim. \end{proof}

\section{Graph classes with bounded strong degeneracy}\label{sec:boundeddeg}

In this section, we analyse classes of graphs with bounded strong degeneracy, and in particular prove Theorem~\ref{thm:char}, which characterises such graph classes.

Let us consider a fixed number $d \in \mathbb{N}$. What prevents a given graph $G$ from being strongly $d$-degenerate? The following are canonical obstacles which could occur.
\begin{itemize}
\item A first possibility is that $G$ contains the complete bipartite graph $K_{2,s}$ for some number $s>d$ as a subgraph. In this case, $K_{2,s}$ contains no $d$-removable vertex, since all its vertices have degree $s>d$, or have two neighbours of degree $s>d$ in their neighbourhood. 
\item A second obstruction of a similar type is if $G$ contains as a subgraph the $1$-subdivision $H^{(1)}$ of a graph $H$ of minimum degree $k>d$. 

Since every vertex in $H^{(1)}$ either has degree at least $k>d$ or is adjacent to two different vertices of degree $k>d$, we can see that $H^{(1)}$ contains no $k$-removable vertex.
\end{itemize}

Given the above two obstructions, it is natural to ask whether graphs excluding both of the obstructions indeed have bounded strong degeneracy. The statement of Theorem~\ref{thm:char} is exactly this, namely, that if a graph does not contain a large $K_{2,s}$-subgraph nor a $1$-subdivision of a graph of large minimum degree, then its strong degeneracy is bounded by a constant. The proof of Theorem~\ref{thm:char} will be done via bounds on the strong degeneracy of a graph in terms of so-called \emph{top-grads} (defined below). In the following, let us start by introducing useful notation and some definitions taken from the book by Ne\v{s}et\v{r}il and Ossona de Mendez about sparsity in graphs~\cite{sparsity}. These have become fairly standard by now in the graph sparsity community, and using them here simplifies the statements of our results significantly and makes it easier to combine them with known results about sparse graph classes.

Let a graph $G$ and an integer $r \ge 0$ be given. We say that a graph $H$ is a \emph{shallow minor} of $G$ at \emph{depth} $r$, if there exists a collection $\mathcal{P}$ of disjoint subsets $V_1,\ldots,V_p$ of $V(G)$ such that:
\begin{itemize}
    \item Each graph $G[V_i], i=1,\ldots,p$ has radius at most $r$: there exists in each set $V_i$ a vertex $x_i \in V_i$ (a \emph{center}) such that every vertex in $V_i$ is at distance at most $r$ from $x_i$ in $G[V_i]$,
    \item $H$ is a subgraph of the graph $G/\mathcal{P}$: each vertex $v$ of $H$ corresponds (in an injective way) to a set $V_{i(v)} \in \mathcal{P}$ and two adjacent vertices $u, v$ of $H$ correspond to two sets $V_{i(u)}$ and $V_{i(v)}$ linked by at least one edge.
\end{itemize}

Similarly, we say that $H$ is a \emph{shallow topological minor} of $G$ at depth $a$ (where $a$ is a non-negative half-integer) if $H$ can be obtained from $G$ by taking a subgraph and then replacing an internally vertex-disjoint family of paths of length at most $2a+1$ by single edges. In other words, $G$ contains as a subgraph a subdivision of $H$ in which every subdivision-path has length at most $2a+1$.
Given an integer $r \ge 0$ and a half-integer $a \ge 0$, we denote by $G \ \grad \ r$ and $G \ \widetilde{\grad} \ a$ the sets of all graphs which are shallow minors of $G$ at depth $r$, or shallow topological minors of $G$ at depth $a$, respectively. Given a non-empty graph $H$, its \emph{density} is defined as $\frac{e(H)}{v(H)}$. Given a graph $G$, an integer $r \ge 0$ and a half-integer $a \ge 0$, the \emph{greatest reduced average density} (shortly \emph{grad}) with rank $r$ of $G$ is defined as the maximum density of graphs in $G \ \grad \ r$, i.e.,
$$\grad_r(G):=\max \left\{\frac{e(H)}{v(H)}\bigg| H \in G \ \grad \ r \right\}.$$

Similarly, the \emph{topological greatest reduced average density} (short \emph{top-grad}) with rank $a$ of $G$ is the maximum density of graphs in $G \ \widetilde{\grad} \ a$, i.e., 
$$\widetilde{\grad}_a(G):=\max \left\{\frac{e(H)}{v(H)}\bigg| H \in G \ \widetilde{\grad} \ a \right\}.$$

We are now ready for proving two different bounds on the strong degeneracy in terms of grads and top-grads, in the form of Propositions~\ref{strongdeg:sufficient} and~\ref{strongdeg:sufficient2}. 

\begin{proposition}\label{strongdeg:sufficient}
Let $s \ge 2$ and let $G$ be a $K_{2,s}$-free graph. Then $G$ is strongly $d$-degenerate, where $d=d(G):=\lfloor(2\widetilde{\grad}_0(G)-2)(s-1)\widetilde{\grad}_\frac{1}{2}(G)+2\widetilde{\grad}_0(G) \rfloor < 2s\widetilde{\grad}_0(G)\widetilde{\grad}_\frac{1}{2}(G) \le 2s \widetilde{\grad}_\frac{1}{2}(G)^2$.
\end{proposition}
\begin{proof}
Towards a contradiction, suppose that there exists a $K_{2,s}$-free graph $G$ which is not strongly $d(G)$-degenerate. Assume further that $G$ is such a graph minimising $v(G)$. 

Note that $G$ does not contain a $d$-removable vertex $v$, for otherwise the graph $G-v$ (by minimality of $G$) would satisfy the assertion of the proposition, and hence be strongly $d(G-v)$-degenerate (and therefore also strongly $d$-degenerate, since $d(G-v) \le d(G)=d$). But then also $G$ would be strongly $d$-degenerate, since every subgraph $G' \subseteq G$ is either a subgraph of $G-v$ (and hence contains a $d$-removable vertex), or contains the $d$-removable vertex $v$. This yields a contradiction to our assumptions about $G$.

Moving on, we may therefore assume that $G$ contains no $d$-removable vertex. Split the vertex-set of $G$ into two sets $A$ and $B$, defined by $A:=\{w \in V(G)|d_G(w) \le d\}$ and $B:=\{w \in V(G)|d_G(w)>d\}$. Observe that every vertex in $A$ has at least two neighbours in $B$, for otherwise it would be $d$-removable. In particular, $B \neq \emptyset$. 
Since $G \in G \ \widetilde{\nabla} \ 0$, we have $e(G) \le \widetilde{\grad}_0(G)v(G)=\widetilde{\grad}_0(G)(|A|+|B|)$. By the handshake-lemma, we further know that 
$$2e(G)=\sum_{w \in V(G)}{d_G(w)} \ge 2|A|+(d+1)|B|.$$
Putting these inequalities together yields $$2|A|+(d+1)|B| \le 2\widetilde{\grad}_0(G)(|A|+|B|),$$ and hence we obtain 
\begin{equation}\label{eq3}(d+1-2\widetilde{\grad}_0(G))|B| \le (2\widetilde{\grad}_0(G)-2)|A|.\end{equation}
Next, for every vertex $a \in A$ select a pair $b_1(a), b_2(a) \in N_G(a) \cap B$ of two distinct neighbours in $B$. Let $H$ be the graph with vertex-set $B$ whose edges are $E(H):=\{b_1(a)b_2(a)|a \in A\}$. It is clear from the definition that the $1$-subdivision $H^{(1)}$ is isomorphic to a subgraph of $G$, and hence $H \in G \ \widetilde{\nabla} \ \frac{1}{2}$. We therefore have $e(H) \le \widetilde{\grad}_\frac{1}{2}(G) v(H)=\widetilde{\grad}_\frac{1}{2}(G)|B|$. Note that since $G$ is $K_{2,s}$-free, for every fixed edge $b_1b_2 \in E(H)$, there exist at most $s-1$ different vertices $a \in A$ such that $\{b_1(a),b_2(a)\}=\{b_1,b_2\}$. As a consequence, we find that
\begin{equation}\label{eq4}|A| \le (s-1)e(H)\le(s-1)\widetilde{\grad}_\frac{1}{2}(G)|B|.\end{equation}
Combining \eqref{eq3} and~\eqref{eq4}, we conclude
$$(d+1-2\widetilde{\grad}_0(G))|B| \le (2\widetilde{\grad}_0(G)-2)(s-1)\widetilde{\grad}_\frac{1}{2}(G)|B|.$$
Since $d+1-2\widetilde{\grad}_0(G)>(2\widetilde{\grad}_0(G)-2)(s-1)\widetilde{\grad}_\frac{1}{2}(G)$ by definition of $d$, this shows that $B=\emptyset$, a contradiction. 
\end{proof}

Equipped with Proposition~\ref{strongdeg:sufficient}, we are ready to present the (short) derivation of Theorem~\ref{thm:char}. 

\begin{proof}[Proof of Theorem~\ref{thm:char}]
Let $\mathcal{G}$ be a class of graphs closed under taking subgraphs. For the first direction of the equivalence, suppose that there exists $d \in \mathbb{N}$ such that every graph $G \in \mathcal{G}$ is strongly $d$-degenerate. Let $s:=d+1$ and $k:=d+1$. We have seen at the beginning of this section that $K_{2,s}$ is not strongly $d$-degenerate, and hence $K_{2,s} \notin \mathcal{G}$. Secondly, for every graph $H$ of minimum degree $k$, we have seen in the beginning of the section that its $1$-subdivision $H^{(1)}$ is not strongly $d$-degenerate, and hence $H^{(1)} \notin \mathcal{G}$, as required. 

For the reversed direction of the equivalence, suppose we are given integers $s, k \in \mathbb{N}$ such that $K_{2,s} \notin \mathcal{G}$ and such that for every graph $H$ of minimum degree at least $k$, we have for the $1$-subdivision $H^{(1)}$ that $H^{(1)} \notin \mathcal{G}$. 

Let now $G \in \mathcal{G}$ be given. Since $\mathcal{G}$ is closed under taking subgraphs, $G$ is $K_{2,s}$-free, and hence it follows from Proposition~\ref{strongdeg:sufficient} that $G$ is strongly $d$-degenerate for some integer $d \le 2s \widetilde{\grad}_\frac{1}{2}(G)^2$. 

Hence, if we can show that there exists an absolute constant $C>0$ such that $\widetilde{\grad}_\frac{1}{2}(G) \le C$ for every $G \in \mathcal{G}$, then it will follow that graphs in $\mathcal{G}$ have bounded strong degeneracy, and the proof of the second part of the equivalence will be complete. 

So let $G \in \mathcal{G}$ given arbitrarily, and let  $\alpha:=\widetilde{\grad}_\frac{1}{2}(G)-1$. Pick a graph $H \in G \ \widetilde{\grad} \ \frac{1}{2}$ such that $\frac{e(H)}{v(H)} \ge \alpha$. By definition of $G \ \widetilde{\grad} \ \frac{1}{2}$, the graph $H$ is obtained from $G$ by first taking a subgraph, and then replacing a set of internally disjoint paths of length two by direct connections between their endpoints. In other words, $G$ contains a subgraph $H^{(t)}$ which can be obtained from $H$ (up to isomorphism) by subdividing a subset $E_1 \subseteq E(H)$ of the edges of $H$ with a single vertex. Write $E_2:=E(H)\setminus E_1$ for the complement of $E_1$ within the edge-set of $H$. Then we have $$\alpha \le \frac{e(H)}{v(H)}=\frac{|E_1|}{v(H)}+\frac{|E_2|}{v(H)}.$$ 

We claim that the spanning subgraph $H_1:=(V(H),E_1)$ of $H$ is $(k-1)$-degenerate. Indeed, suppose not, then there exists a subgraph $F \subseteq H_1$ of minimum degree at least $k$. Then $G$ contains $H^{(t)}$ as a subgraph, which by definition of $E_1$ in turn contains the $1$-subdivision of $H_1$ and hence the $1$-subdivision of $F$ as a subgraph (up to isomorphism). Hence, a graph isomorphic to a $1$-subdivision of $F$ is contained in $\mathcal{G}$, which contradicts the facts that $F$ has minimum degree at least $k$ and our initial assumptions on the class $\mathcal{G}$. 
Hence, we have established that $H_1$ is $(k-1)$-degenerate, and it follows that $\frac{|E_1|}{v(H)}=\frac{e(H_1)}{v(H_1)} \le k-1$. 

Next, we claim that the spanning subgraph $H_2:=(V(H),E_2)$ of $H$ has average degree less than $M\cdot (k+1)^2\log(k+1)$ for some absolute constant $M>0$. Indeed, this follows (for instance) from Lemma~1 in~\cite{dvorak}, which implies that for some absolute constant $M>0$, every graph of average degree at least $M\cdot (k+1)^2\log(k+1)$ contains as a subgraph the $1$-subdivision of a graph of chromatic number at least $k+1$, and hence also a $1$-subdivision of a graph of minimum degree at least $k$. However, since $H_2$ (by definition of $E_2$) is isomorphic to a subgraph of $H^{(t)}$, which in turn is included in $G$, it follows by our initial assumptions on $\mathcal{G}$ that $H$ does not contain a $1$-subdivision of a graph of minimum degree at least $k$, establishing the above intermediate claim. 

Hence, we have $\frac{|E_2|}{v(H_2)}=\frac{1}{2}\frac{2e(H_2)}{v(H_2)} < \frac{M}{2}(k+1)^2\log(k+1)$. 

Putting everything together, we find that
$$\widetilde{\grad}_\frac{1}{2}(G)-1 = \alpha \le (k-1)+\frac{M}{2}(k+1)^2\log(k+1).$$
Since our choice of $G \in \mathcal{G}$ was arbitrary, it follows that $C:=k+\frac{M}{2}(k+1)^2\log(k+1)$ is a constant for which we have $\widetilde{\grad}_\frac{1}{2}(G) \le C$ for every $G \in \mathcal{G}$. As discussed above, this implies that there exists a constant $d \in \mathbb{N}$ such that every graph in $\mathcal{G}$ is strongly $d$-degenerate. This concludes the proof of the second direction of the equivalence, and hence also the proof of Theorem~\ref{thm:char}.
\end{proof}

While Proposition~\ref{strongdeg:sufficient} bounds the strong degeneracy of a $K_{2,s}$-free graph by a quadratic function of its top-grad at depth $\frac{1}{2}$, in the next proposition, we show that the strong degeneracy of a $K_{2,s}$-free graph can also be bounded by a \emph{linear} function of its grad at depth $1$.

\begin{proposition}\label{strongdeg:sufficient2}
Let $s \ge 2$ and let $G$ be a $K_{2,s}$-free graph. Then $G$ is strongly $d$-degenerate, where $d=d(G):=\lfloor 2s\grad_1(G) \rfloor$. 
\end{proposition}
\begin{proof}
As above, suppose towards a contradiction that there exists a $K_{2,s}$-free graph $G$ which is not strongly $d(G)$-degenerate, and assume that $G$ minimizes $v(G)$ with respect to these properties. 
Just as in the proof of Proposition~\ref{strongdeg:sufficient}, using that $d(G-v) \le d(G)$ for every $v \in V(G)$ and the minimality of $G$, we find that $G$ contains no $d$-removable vertex. Let $A:=\{w \in V(G)|d_G(w) \le d\}$ and $B:=\{w \in V(G)|d_G(w)>d\}$, note that every vertex in $A$ has at least two neighbours in $B$, and hence $B \neq \emptyset$. 
Next, for every $a \in A$, select and fix (arbitrarily) two distinct neighbours $b(a), b'(a) \in B$ of $a$ in $G$. Let $H$ be the graph with vertex-set $B$ and whose edges are \[E(H):=E(G[B]) \cup \{b(a)b|a \in A, b \in B, ab \in E(G), b(a) \neq b\}.\] We claim that $H \in G \ \grad \ 1$. Indeed, it is readily verified that $H$ can be obtained from $G$ by first deleting all edges in $G[A]$, and then, for every $b \in B$, contracting $X_b:=\{b\}\cup\{a \in A|b(a)=b\}$ into a vertex (note that $G[X_b]$ is a star centered at $b$ and hence has radius $1$). From this, we may conclude that 

\begin{equation}\label{eq5}e(H) \le \grad_1(G) v(H)=\grad_1(G) |B|.\end{equation}

Let us next bound the minimum degree of $H$. So consider any vertex $b \in B$. Let us define a mapping $\phi:N_G(b) \rightarrow N_H(b)$, by putting $\phi(x):=x$ for every $x \in N_G(b) \cap B$,  $\phi(x):=b(x)$ for every $x \in N_G(b) \cap A$ such that $b(x) \neq b$, and $\phi(x):=b'(x)$ for every $x \in N_G(b) \cap A$ such that $b(x)=b$. Pause to note that this mapping is well-defined.

Next consider any fixed vertex $c \in N_H(b)$. Then we have by definition of $\phi$:
$$\phi^{-1}(c) \subseteq \{c\} \cup \{a \in N_G(b) \cap A|b(a)=c\} \cup \{a \in N_G(b) \cap A|b'(a)=c, b(a)=b\}$$ $$\subseteq \{c\} \cup \{a \in A|ab, ac \in E(G)\}.$$
Since $G$ contains no $K_{2,s}$-subgraph, we can see from this that $|\phi^{-1}(c)| \le 1+(s-1)=s$. Since $c \in N_H(b)$ was chosen arbitrarily, it follows that $d+1 \le d_G(b)=\sum_{c \in N_H(b)}{|\phi^{-1}(c)|} \le s d_H(b)$ for every $b \in B$. Summing this inequality over all $b \in B$, we have
\begin{equation}\label{eq6}(d+1)|B| \le s\sum_{b \in B}{d_H(b)}=2s\cdot e(H).\end{equation}
Combining with~\eqref{eq5}, we find that $(d+1)|B| \le 2s\grad_1(G)|B|$. Since $d+1>2s\grad_1(G)$, this yields $B=\emptyset$, a contradiction. This shows that our initial assumption concerning the existence of $G$ was wrong, and concludes the proof.
\end{proof}

Proposition~\ref{strongdeg:sufficient} and Proposition~\ref{strongdeg:sufficient2} allow us to deduce that many sparse classes of graphs have bounded strong degeneracy, and hence (by Theorem~\ref{thm:hattiness}) also have a constant bound on their hat guessing number. In particular, below we can now give the (short) proofs of Corollaries~\ref{hadwiger} and~\ref{topological}. 

\begin{proof}[Proof of Corollary~\ref{hadwiger}]
It was proved independently by Kostochka~\cite{kostochka} and Thomason~\cite{thomason} that there exists a constant $C>0$ such that every graph $H$ containing no $K_t$-minor satisfies $\frac{e(H)}{v(H)} \le Ct\sqrt{\log t}$. It follows easily from the definition that for every $K_{2,s}$-free graph $G$ without a $K_t$-minor, also all the graphs in $G \ \grad \ r$ do not contain a $K_t$-minor, for every integer $r \ge 0$. Hence, $\grad_r(G) \le Ct\sqrt{\log t}$ for every choice of $r$, and Proposition~\ref{strongdeg:sufficient2} yields that $G$ is strongly $d$-degenerate for every integer $d$ such that $d \ge 2s \grad_1(G)=O(st\sqrt{\log t})$. This proves the claim (the upper bound on the hat guessing number follows directly from Theorem~\ref{thm:hattiness}). 
\end{proof}
\begin{proof}[Proof of Corollary~\ref{topological}]
It was proved independently by Koml\'{o}s and Szemer\'{e}di~\cite{komlos} and Bollob\'{a}s and Thomason~\cite{bollobas} that there exists a constant $C>0$ such that every graph $H$ containing no $K_t$-subdivision satisfies $\frac{e(H)}{v(H)} \le Ct^2$. It follows easily from the definition that for every $K_{2,s}$-free graph $G$ without a $K_t$-subdivision, also all the graphs in $G \ \widetilde{\grad} \ a$ do not contain a $K_t$-subdivision, for every half-integer $a \ge 0$. Hence, $\widetilde{\grad}_a(G) \le Ct^2$ for every choice of $a$, and Proposition~\ref{strongdeg:sufficient} yields that $G$ is strongly $d$-degenerate for every integer $d$ such that $d \ge 2s \widetilde{\grad}_\frac{1}{2}(G)^2=O(st^4)$. This proves the claim (again, the upper bound on the hat guessing number follows from Theorem~\ref{thm:hattiness}). 
\end{proof}

A central notion in the book of Ne\v{s}et\v{r}il and Ossona de Mendez~\cite{sparsity} constitutes the concept of graph classes of \emph{bounded expansion} (we also refer to~\cite{charac1}, and the enormous amount of other articles written on the topic of bounded expansion, which we cannot list here, for further background). A graph class $\mathcal{G}$ is said to have \emph{bounded expansion} if for every integer $r \ge 0$, there exists a constant $C_r$ such that every graph $G\in\mathcal{G}$ satisfies $\grad_r(G) \le C_r$. Equivalently (compare~\cite{sparsity,charac1}), $\mathcal{G}$ has bounded expansion if and only if for every half-integer $a \ge 0$ there exists a constant $C_a'$ such that $\widetilde{\grad}_a(G) \le C_{a}'$ for every $G \in \mathcal{G}$. 

Graph classes of bounded expansion unify several well-studied classes of sparse graphs, such as graphs with no $K_t$-subdivision, graphs with no $K_t$-minor (as addressed in the previous two corollaries), graphs of bounded maximum degree, string graphs with forbidden complete bipartite subgraphs, sparse random graphs, and many more, cf.~\cite{charac2,sparsity,charac1,charac3}. Since in particular, the (top-)grads $\grad_1(G), \widetilde{\grad}_0(G), \widetilde{\grad}_{\frac{1}{2}}(G)$ are bounded by a constant for all members $G$ of a graph-class with bounded expansion, the claim of Corollary~\ref{boundedexpansion} immediately follows by combining optionally Proposition~\ref{strongdeg:sufficient} or Proposition~\ref{strongdeg:sufficient2} with Theorem~\ref{thm:hattiness}. 

\begin{proof}[Proof of Corollary~\ref{boundedexpansion}]
Let $\mathcal{G}$ be a class of graphs with bounded expansion, and let $s \in \mathbb{N}$. Then there exists a constant $C_1>0$ such that every $G \in \mathcal{G}$ satisfies $\grad_1(G) \le C_1$.  By Proposition~\ref{strongdeg:sufficient2}, we know that every $K_{2,s}$-free graph $G \in \mathcal{G}$ is strongly $d$-degenerate where $d:=\lfloor 2s\grad_1(G) \rfloor \le \lfloor 2sC_1\rfloor$, and hence, the maximum hat guessing number of $K_{2,s}$-free graphs in $\mathcal{G}$ is bounded by a function of $s$ and $\mathcal{G}$ acoording to Theorem~\ref{thm:hattiness}. This verifies the claim. 
\end{proof}

A nice application of Corollary~\ref{boundedexpansion} is that it allows us to give a short proof of Corollary~\ref{random}, stating that for every fixed value $C>0$, the hat guessing number of sparse Erd\H{o}s-Renyi-random graphs $G(n,\frac{C}{n})$ is upper-bounded by a constant solely dependent on $C$, w.h.p.. 

\begin{proof}[Proof of Corollary~\ref{random}]
Let us start by noting that w.h.p. as $n \rightarrow \infty$, the random graph $G(n,\frac{C}{n})$ is $K_{2,3}$-free:
Indeed, since $v(K_{2,3})=5, e(K_{2,3})=6$, the expected number of copies of $K_{2,3}$ in $G(n,\frac{C}{n})$ is upper-bounded by 
$$n^5\left(\frac{C}{n}\right)^6=\frac{C^6}{n} \rightarrow 0$$ as $n \rightarrow \infty$, and hence the claim follows by the first-moment principle. 

Second, it was proved by Ne\v{s}et\v{r}il and Ossona de Mendez in~\cite[Chapter 14.1]{sparsity} that there exists a bounded-expansion class $\mathcal{G}_C$ of graphs (depending solely on the choice of $C$) such that $G(n,\frac{C}{n})$ is a member of $\mathcal{G}_C$ w.h.p. as $n \rightarrow \infty$. The claim now follows directly from Corollary~\ref{boundedexpansion}. 
\end{proof}

\section{Open Problems}

Apart from the big challenge of answering Problem~\ref{problems} concerning whether the parameters \emph{hat guessing number} and \emph{degeneracy} are qualitatively tied to one another, several more special natural questions regarding the behaviour of the hat guessing number on certain graph classes have come up during our work on this problem, and we would like to share some of them. 

We have proved in Theorem~\ref{thm:outplanar} that every outerplanar graph $G$ satisfies $\text{HG}(G) \le 40$, and an example called ``trefoil'' constructed in~\cite{kokhas2021hats} shows that there exist outerplanar graphs with hat guessing number~$6$. 
\begin{question}
What is the smallest integer $C$ such that every outerplanar graph $G$ satisfies $\text{HG}(G) \le C$?
\end{question}

As we have shown in Corollary~\ref{boundedexpansion}, all $C_4$-free graphs (or, more generally, $K_{2,s}$-free graphs for fixed~$s$) that belong to a sparse class of graphs (e.g., with bounded expansion) have a hat guessing number which is upper bounded by a universal constant (depending only on the class of graphs). We do suspect that the requirement of ``sparsity'' of the graphs in our results is indeed necessary, but yet we have failed to come up with examples of $C_4$-free graphs that have unbounded hat guessing numbers. We therefore ask the following question, which was also previously posed by Bradshaw in~\cite{bradshaw2021hat}.

\begin{question}
Does there exist, for every given $C>0$, a $C_4$-free graph $G$ such that $\text{HG}(G) \ge C$?
\end{question}

Note that forbidding odd cycles as subgraphs does not bound the hat guessing number, since complete bipartite graphs can have arbitrarily large hat guessing numbers~\cite{alon2020hat}.
One natural approach towards the above question could be to study the hat guessing numbers of the incidence graphs of projective planes, which are known to be dense $C_4$-free graphs. The above would be answered in the positive once the following problem which was first proposed in~\cite{he2020hat} and reiterated in~\cite{bradshaw2021hat} is confirmed:

\begin{question}
Does there exist, for every given $g \in \mathbb{N}$ and $C>0$, a graph $G$ of girth at least $g$ with $\text{HG}(G) \ge C$?
\end{question}

Since constructions of dense graphs of large girth are often based on randomness, in order to approach the previous problem, it might be rewarding to obtain a broader understanding of the hat guessing number of the Erd\H{o}s-Renyi random graphs $G(n,p)$ for a wide range of values for $p$, for instance, if $\frac{1}{n}\ll p \ll 1$. 
In view of Problem~\ref{problems}, (2.), we think it would be particularly interesting to answer the following question. 

\begin{question}
Let $p=p(n)$ be such that $np(n) \rightarrow \infty$. Is it true that for every constant $C>0$, we have $\text{HG}(G(n,p)) > C$ w.h.p.? 
\end{question}

We note that as the hat guessing number is bounded from below by the clique number, the statement holds true for any $p\geq n^{-o(1)}$, but for any $\frac{1}{n} \ll p \leq n^{-\varepsilon}$ for a fixed $\varepsilon>0$ the problem is wide open.

Many results on the hat guessing number seem to indicate a behaviour that is similar in some regards to that of the chromatic number of graphs. This inspired Bosek et al.~\cite{bosek2021hat} to conjecture the following relation between the parameters, which we reiterate here.

\begin{conjecture}
Every graph $G$ satisfies $\chi(G) \le \text{HG}(G)+1$. 
\end{conjecture}

\section*{Acknowledgements}

This work was initiated at the research workshop of Angelika Steger in Buchboden, August 2021. We thank Andreas Noever for fruitful discussions.

\bibliographystyle{abbrv}
\bibliography{sources}
\end{document}